\documentclass[a4paper,11pt]{article}
\usepackage[bookmarks=true,
              bookmarksnumbered=true, breaklinks=true,
              pdfstartview=FitH, hyperfigures=false,
              plainpages=false, naturalnames=true,
              colorlinks=true,
              pdfpagelabels]{hyperref}
\usepackage{geometry,latexsym,amssymb,amsmath,amsthm,color,bm}
\usepackage[latin5]{inputenc}
\usepackage{enumerate}
\usepackage{enumitem}
\usepackage[T1]{fontenc}
\usepackage{authblk}
\usepackage[all]{xy}
\usepackage{palatino}
\usepackage{indentfirst}
\usepackage{titlesec}
\usepackage{graphics}
\usepackage{tabularx}
\usepackage{lipsum}

\theoremstyle{plain}

\theoremstyle{definition}

\newtheorem{thm}{Theorem}[section]
\newtheorem{cor}[thm]{Corollary}
\newtheorem{lem}[thm]{Lemma}
\newtheorem{prop}[thm]{Proposition}
\newtheorem{defn}[thm]{Definition}

\newtheorem{ex}[thm]{Example}
\def\XMod{\bm{\mathsf{XMod}}}
\def\XSq{\bm{\mathsf{X^2Mod}}}
\def\GpGd{\bm{\mathsf{GpGd}}}
\def\Cat{\bm{\mathsf{Cat}}}

\def\Cat{\bm{\mathsf{Cat}}}

\def\GpGd{\bm{\mathsf{GpGd}}}
\def\XMod{\bm{\mathsf{XMod}}}

\def\XMod{\bm{\mathsf{XMod}}}
\def\GpGd{\bm{\mathsf{GpGd}}}

\def\Cat{\bm{\mathsf{Cat}}}
\def\XSq{\bm{\mathsf{X^2Mod}}}

\def\GGdC/G{\bm{\mathsf{GpGpdCov/G}}}
\def\GGdCov/X{\bm{\mathsf{GpGpdCov/\pi X}}}
\def\GdC/G{\bm{\mathsf{GpdCov/G}}}
\def\GdA(G){\bm{\mathsf{GpdAct(G)}}}
\def\Act(G){\bm{\mathsf{GpdAct(G)}}}
\def\Cov/G{\bm{\mathsf{GpdCov/G}}}

\def\epsilon{\varepsilon}

\begin{document}
\title{Categories internal to crossed modules}

\author[a]{Tunçar ŞAHAN\thanks{T. Şahan (e-mail : tuncarsahan@gmail.com)}}
\author[b]{Jihad Jamil Mohammed\thanks{J.J. Mohammed (e-mail : cihad.cemil@gmail.com)}}
\affil[a]{\small{Department of Mathematics, Aksaray University, Aksaray, TURKEY}}
\affil[b]{\small{Wan Primary School, Zakho, Iraq}}

\date{}

\maketitle

\begin{abstract}
In this study, internal categories in the category of the crossed modules are characterized and it has been shown that there is a natural equivalence between the category of the crossed modules over crossed modules, i.e. crossed squares, and the category of the internal categories within the category of crossed modules. Finally, we obtain examples of crossed squares using this equivalence.\end{abstract}

\\ {\bf Classification:} 18D35, 18G50, 20J15, 20L05.

\section{Introduction }
Crossed modules are first defined in the works of Whitehead \cite{Wh46,Wh48,Wh49} and has been found important in many areas of mathematics including homotopy theory, group representation theory, homology and cohomology on groups, algebraic K-theory, cyclic homology, combinatorial group theory and differential geometry. See \cite{Br84,BHi78,BHu82,BL87} for applications of crossed modules.  Later, it was shown that the categories of the internal categories in the category of groups and the category of the crossed modules are equivalent \cite{BS76,Lod82}.

Mucuk et al. \cite{MSA15} interpret the concept of normal subcrossed module and quotient crossed module concepts in the category of internal categories within groups, that is group-groupoids. The equivalences of the categories given in \cite[Theorem 1]{BS76} and \cite[Section 3]{Por87} enable to generalize some results on group-groupoids to the more general internal groupoids for an arbitrary category of groups with operations (see for example \cite{AAMS13}, \cite{MA15}, \cite{MKSA11} and \cite{MS14}).

Lichtenbaum, Schlessinger \cite{LS67}, and Gerstenhaber \cite{Ge66} have defined the concept of a crossed module on associative and commutative algebras. In \cite{Ar97} the categories of crossed modules and of 2-crossed modules on commutative algebras are linked with an equivalence.

Crossed square is first described to be applied to algebraic K-theoretic problems \cite{GWLod81}. Crossed squares are two-dimensional analogous of crossed modules and model all connected homotopy 3-types (hence all 3-groups) and correspond in much the same way to pairs of normal subgroups while crossed modules model all connected homotopy 2-types and groups model all connected homotopy 1-types.

Recently, freeness conditions for 2-crossed modules and crossed squares are given in \cite{MPor98} and \cite{MPor00}. See also \cite{AP88} for commutative algebra case.

Main objective of this study is to characterize the internal categories within the category of crossed modules and to prove that the category of the internal categories in the category of crossed modules and the category of crossed squares are equivalent. Hence this equivalence allow us to produce more examples of crossed squares.

\section{Preliminaries}

\subsection{Extensions and crossed modules}

Following are detailed descriptions of the ideas given in \cite{Por87} for the case of groups. An exact sequence of the form
\[\xymatrix{
	\bm{\mathsf{0}} \ar[r] &   A \ar@{->}[r]^-{i}  &   E \ar@{->}[r]^-{p} &   B \ar[r]  & \bm{\mathsf{0}} }\]
is called \textit{short exact sequence} where $\bm{\mathsf{0}}$ is the group with one element. Here $i$ is a monomorphism, $p$ is an epimorphism and $\ker p=A$. In a short exact sequence the group $E$ is called an extension of $B$ by $A$. An extension is called \textit{split} if there exist a group homomorphism $s:B\to E$ such that $ps={{1}_{B}}$.

Let $E$ be a split extension of $B$ by $A$. Then the function
\[\begin{array}{rcccl}
\theta & : & E & \rightarrow  & A\times B  \\
&  & e & \mapsto  & \left( e-sp\left( e \right),p\left( e \right) \right)  \\
\end{array}\]
is a bijection. The inverse of $\theta$ is given by ${\theta }^{-1}\left( a,b \right)=a+s\left( b \right)$.

Thus we can define a group structure on $A\times B$ such that $\theta $ is an isomorphism of groups. Let $\left( a,b \right),\left( {{a}_{1}},{{b}_{1}} \right)\in A\times B$. Then
\[\begin{array}{rl}
\left( a,b \right)+\left( {{a}_{1}},{{b}_{1}} \right) & =\theta \left( {{\theta }^{-1}}\left( \left( a,b \right)+\left( {{a}_{1}},{{b}_{1}} \right) \right) \right) \\
& =\theta \left( {{\theta }^{-1}}\left( a,b \right)+{{\theta }^{-1}}\left( {{a}_{1}},{{b}_{1}} \right) \right) \\
& =\theta \left( a+s\left( b \right)+{{a}_{1}}+s\left( {{b}_{1}} \right) \right) \\
& =\left( a+s\left( b \right)+{{a}_{1}}+s\left( {{b}_{1}} \right)-s\left( {{b}_{1}} \right)-s\left( b \right),b+{{b}_{1}} \right) \\
& =\left( a+\left( s\left( b \right)+{{a}_{1}}-s\left( b \right) \right),b+{{b}_{1}} \right).
\end{array}\]

$A\times B$ is called the semi-direct product group of $A$ and $B$ with the operaton given above and denoted by $A\rtimes B$. Here we note that a split extension of $B$ by $A$ defines an (left) action of $B$ on $A$ with
\[b\cdot a=s\left( b \right)+a-s\left( b \right)\] for $a\in A$ and $b\in B$.

These kind of actions are called derived actions \cite{Orz72}. Every group $A$ has a split extension by itself in a natural way which gives rise to the conjugation action as
\[\xymatrix{
	\bm{\mathsf{0}} \ar[r] &   A \ar@{->}[r]^-{i}  &   A\rtimes   A \ar@{->}[r]_-{p} &   A \ar@/_/[l]_-{s} \ar[r]  & \bm{\mathsf{0}} }\]
where $i\left( a \right)=\left( a,0 \right)$, $p\left( a,{{a}_{1}} \right)={{a}_{1}}$ and $s\left( a \right)=\left( 0,a \right)$ for $a,{{a}_{1}}\in A$.

\begin{defn}
	Let $A$ and $B$ two groups and let $B$ acts on $A$ on the left. Then a group homomorphism $\alpha :A\to B$ is called a crossed module if ${{1}_{A}}\times \alpha \colon A\rtimes A\rightarrow A\rtimes B$ and $\alpha \times {{1}_{B}}\colon A\rtimes B\rightarrow B\rtimes B$ are group homomorphisms\cite{Por87}.
\end{defn}

A crossed module is denoted by $\left( A,B,\alpha  \right)$. It is useful to give the definition of crossed modules in terms of group operations and actions.

\begin{prop}\label{propxmod}
	Let $A$ and $B$ be two groups, $\alpha :A\to B$ a group homomorphism and $B$ acts on $A$. Then $\left( A,B,\alpha  \right)$ is a crossed module if and only if
	\begin{enumerate}[label={\textbf{(CM \arabic{*})}}, leftmargin=2cm]
		\item $\alpha \left( b\cdot a \right)=b+\alpha \left( a \right)-b$ and
		\item $\alpha \left( a \right)\cdot {{a}_{1}}=a+{{a}_{1}}-a$
	\end{enumerate}	
	for all $a,{{a}_{1}}\in A$ and $b\in B$ \cite{Por87}.
\end{prop}

\begin{ex}
	Following homomorphisms are standart examples of crossed modules.
	\begin{enumerate}[label={\textbf{(\roman{*})}}]
		\item Let $X$ be a topological space, $A\subset X$ and $x\in A$. Then the boundary map $\rho $ from the second relative homotopy group ${{\pi }_{2}}\left( X,A,x \right)$ to the fundamental group ${{\pi }_{1}}\left( X,x \right)$ is a crossed module.
		\item Let $G$ be a group and $N$ a normal subgroup of $G$. Then the inclusion function $N\xrightarrow{inc}G$ is a crossed module where the action of $G$ on $N$ is conjugation.
		\item Let $G$ be a group. Then the inner automorphism map $G\rightarrow \operatorname{Aut}(G)$ is a crossed module. Here the action is given by $\psi\cdot g=\psi(g)$ for all $\psi\in\operatorname{Aut}(G)$ and $g\in G$.
		\item Given any $G$-module, $M$, the trivial homomorphism $0:M\to G$ is a crossed $G$-module with the given action of $G$ on $M$.
	\end{enumerate}
\end{ex}

A morphism $f=\left\langle {{f}_{A}},{{f}_{B}} \right\rangle $ of crossed modules from $\left( A,B,\alpha  \right)$ to $\left( A',B',\alpha ' \right)$ is a pair of group homomorphisms ${{f}_{A}}:A\to A'$ and ${{f}_{B}}:B\to B'$ such that ${{f}_{B}}\alpha =\alpha '{{f}_{A}}$ and ${{f}_{A}}\left( b\cdot a \right)={{f}_{B}}\left( b \right)\cdot {{f}_{A}}\left( a \right)$ for all $a\in A$ and $b\in B$.
\[\xymatrix{B\times A \ar[r]^-{\cdot} \ar[d]_{f_B \times f_A} & A \ar[r]^{\alpha} \ar[d]_{f_A}  &  B \ar[d]^{f_B} \\ B'\times A' \ar[r]_-{\cdot} & A' \ar[r]_{\alpha'} & B' }\]

Crossed modules form a category with morphisms defined above. The category of crossed modules is denoted by $\XMod$.

\begin{defn}\cite{Nor87,Nor90}
	Let $\left( A,B,\alpha  \right)$ and $\left( S,T,\sigma  \right)$ be two crossed modules. Then $\left( S,T,\sigma  \right)$ is called a subcrossed module of $\left( A,B,\alpha  \right)$ if $S\le A$, $T\le B$, $\sigma $ is the restriction of $\alpha $ to $S$ and the action of $T$ on $S$ is the induced action from that of $B$ on $A$.
\end{defn}

\begin{defn}\cite{Nor87,Nor90}
	Let $\left( A,B,\alpha  \right)$ be a crossed module and $\left( S,T,\sigma  \right)$ a subcrossed module of $\left( A,B,\alpha  \right)$. Then $\left( S,T,\sigma  \right)$ is called a normal subcrossed module or an ideal of $\left( A,B,\alpha  \right)$ if
	\begin{enumerate}[label={\textbf{(\roman{*})}}]
		\item $T\triangleleft B$,
		\item $b\cdot s\in S$ for all $b\in B$, $s\in S$ and
		\item $t\cdot a-a\in S$ for all $t\in T$,$a\in A$.
	\end{enumerate}
\end{defn}

\begin{ex}
	Let $f=\left\langle {{f}_{A}},{{f}_{B}} \right\rangle :\left( A,B,\alpha  \right)\to \left( A',B',\alpha ' \right)$ be a morphism of crossed modules. Then the kernel $\ker f=\ker \left\langle {{f}_{A}},{{f}_{B}} \right\rangle =\left( \ker {{f}_{A}},\ker {{f}_{B}},{{\alpha }_{|\ker {{f}_{A}}}} \right)$ of $f=\left\langle {{f}_{A}},{{f}_{B}} \right\rangle $ is a normal subcrossed module (ideal) of $\left( A,B,\alpha  \right)$. Moreover, the image $\operatorname{Im}f=\operatorname{Im}\left\langle {{f}_{A}},{{f}_{B}} \right\rangle =\left( \operatorname{Im}{{f}_{A}},\operatorname{Im}{{f}_{B}},\alpha {{'}_{|\operatorname{Im}{{f}_{A}}}} \right)$ of $f=\left\langle {{f}_{A}},{{f}_{B}} \right\rangle $ is a subcrossed module of $\left( A',B',\alpha ' \right)$.
\end{ex}

\begin{defn}
	A topological crossed module $\left( A,B,\alpha  \right)$ consist of two topological groups $A$ and $B$, a continuous group homomorphism $\alpha :A\to B$ and a continuous action of $B$ on $A$ such that the conditions of Proposition \ref{propxmod} are satisfied.
\end{defn}

Now we give the pullback notion in the category of crossed modules.

\begin{defn}
	Let $\left( A,B,\alpha  \right)$, $\left( M,P,\mu  \right)$ and $\left( C,D,\gamma  \right)$ be three crossed modules and $f=\left\langle {{f}_{A}},{{f}_{B}} \right\rangle :\left( A,B,\alpha  \right)\to \left( M,P,\mu  \right)$ and $g=\left\langle {{g}_{C}},{{g}_{D}} \right\rangle :\left( C,D,\gamma  \right)\to \left( M,P,\mu  \right)$ be two crossed module morphisms. Then the pullback crossed module of $f$ and $g$ is $\left( A{}_{{{f}_{A}}}{{\times }_{{{g}_{C}}}}C,B{}_{{{f}_{B}}}{{\times }_{{{g}_{D}}}}D,\alpha \times \gamma  \right)$ where the action of $B{}_{{{f}_{B}}}{{\times }_{{{g}_{D}}}}D$ on $A{}_{{{f}_{A}}}{{\times }_{{{g}_{C}}}}C$ is given by
	\[\left( b,d \right)\cdot \left( a,c \right)=\left( b\cdot a,d\cdot c \right)\]
	for all $\left( b,d \right)\in B{}_{{{f}_{B}}}{{\times }_{{{g}_{D}}}}D$ and $\left( a,c \right)\in A{}_{{{f}_{A}}}{{\times }_{{{g}_{C}}}}C$.
\end{defn}

\subsection{Internal categories and Brown-Spencer Theorem}

\begin{defn}\label{intcat}
	Let $\mathbb{C}$ be a category with pullbacks. Then an internal category $C$ in $\mathbb{C}$ consist of two objects ${{C}_{1}}$ and ${{C}_{0}}$ in $\mathbb{C}$ and four structure morphisms $s,t:{{C}_{1}}\to {{C}_{0}}$, $\varepsilon :{{C}_{0}}\to {{C}_{1}}$ and $m:{{C}_{1}}\,{}_{s}{{\times }_{t}}\,{{C}_{1}}\to {{C}_{1}}$, where ${{C}_{1}}\,{}_{s}{{\times }_{t}}\,{{C}_{1}}$ is the pullback of $s$ and $t$, such that the following conditions hold:
	\begin{enumerate}[label={\textbf{(\roman{*})}}]
		\item\label{ic1} $s\varepsilon =t\varepsilon ={{1}_{{{C}_{0}}}}$;
		\item\label{ic2} $sm=s{{\pi }_{2}}$, $tm=t{{\pi }_{1}}$ ;
		\item\label{ic3} $m\left( {{1}_{{{C}_{1}}}}\times m \right)=m\left( m\times {{1}_{{{C}_{1}}}} \right)$ and
		\item\label{ic4} $m\left( \varepsilon s{{,1}_{{{C}_{1}}}} \right)=m\left( {{1}_{{{C}_{1}}}},\varepsilon t \right)={{1}_{{{C}_{1}}}}$.
	\end{enumerate}	
\end{defn}

Morphisms $s,t,\varepsilon $, and $m$ are called source, target, identity object maps and composition respectively. An internal category in $\mathbb{C}$ will be denoted by $C=\left( {{C}_{1}},{{C}_{0}},s,t,\varepsilon ,m \right)$ or only by $C$ for short.

If there is a morphism $n:{{C}_{1}}\to {{C}_{1}}$ in $\mathbb{C}$ such that
$m\left( 1,n \right)=\varepsilon s$  and  $m\left( n,1 \right)=\varepsilon t$, i.e. every morphism in $C_1$ has an inverse up to the composition, 
then we say that $C=\left( {{C}_{1}},{{C}_{0}},s,t,\varepsilon ,m,n \right)$ is an internal groupoid in $\mathbb{C}$.

Let $C$ and $C'$ be two internal categories in $\mathbb{C}$. Then a morphism $f=\left( {{f}_{1}},{{f}_{0}} \right)$ from $C$ to $C'$ consist of a pair of morphisms ${{f}_{1}}:{{C}_{1}}\to {{C}_{1}}'$ and ${{f}_{0}}:{{C}_{0}}\to {{C}_{0}}'$ in $\mathbb{C}$ such that
\begin{enumerate}[label={\textbf{(\roman{*})}}]
	\item $s{{f}_{1}}={{f}_{0}}s$, $t{{f}_{1}}={{f}_{0}}t$,
	\item $\varepsilon {{f}_{0}}={{f}_{1}}\varepsilon $ and
	\item $m\left( {{f}_{1}}\times {{f}_{1}} \right)={{f}_{1}}m$.
\end{enumerate}

Thus one can construct the category of internal categories in an arbitrary category $\mathbb{C}$ with morphisms defined above. This category is denoted by $\Cat\left( \mathbb{C} \right)$.

An internal category in the category of groups is called a group-groupoid \cite{BS76}. Group-groupoids are also the group objects in the category of small categories. 

\begin{ex}
	Let $X$ be a topological group. Then the set $\pi X$ of all homotopy classes of paths in $X$ defines a groupoid structure on the set of objects $X$. This groupoid is called the fundamental groupoid of $X$. Moreover, $\pi X$ is a group-groupoid \cite{BS76}. 
\end{ex}

Let $G$ be an internal category in the category of groups, i.e. a group-groupoid. Then the object of morphisms ${{G}_{1}}$ and object of objects ${{G}_{0}}$ have group structures and there are four group homomorphisms $s,t:{{G}_{1}}\to {{G}_{0}}$, $\varepsilon :{{G}_{0}}\to {{G}_{1}}$ and $m:{{G}_{1}}\,{}_{s}{{\times }_{t}}\,{{G}_{1}}\to {{G}_{1}}$ such that the conditions \ref{ic1}-\ref{ic4} of Definition \ref{intcat} are satisfied. 

Since $m:{{G}_{1}}\,{}_{s}{{\times }_{t}}\,{{G}_{1}}\to {{G}_{1}}$ is a group homomorphism then we can give the following lemma.

\begin{lem}\label{leminterchange}
	Let $G$ be an internal category in the category of groups. Then
	\[m\left( \left( b,a \right)+\left( b',a' \right) \right)=m\left( \left( b',a' \right) \right)+m\left( \left( b',a' \right) \right),\] i.e.
	\[\left( b+b' \right)\circ \left( a+a' \right)=\left( b\circ a \right)+\left( b'\circ a' \right)\] whenever one side (hence both sides) make senses, for all $a,a',b,b'\in {{G}_{1}}$.
\end{lem}

Morphisms between group-groupoids are functors which are group homomorphisms. The category of group-groupoids is denoted by $\GpGd$.

Equation given in Lemma \ref{leminterchange} is called the interchange law. Applications of interchange law can be given as in the following.

Let $G$ be a group-groupoid. Then the partial composition in $G$ can be given in terms of group operations \cite{BS76}. Indeed, let $a\in G\left( x,y \right)$ and $b\in G\left( y,z \right)$. Then
\[\begin{array}{rl}
b\circ a=& \left( b+0 \right)\circ \left( {{1}_{y}}+\left( -{{1}_{y}}+a \right) \right) \\
& =\left( b\circ {{1}_{y}} \right)+\left( 0\circ \left( -{{1}_{y}}+a \right) \right) \\
& =b-{{1}_{y}}+a
\end{array}\]
and similarly $b\circ a=a-{{1}_{y}}+b$.

\begin{cor}\cite{BS76}
	Let $G$ be a group-groupoid. Then the elements of $\ker s$ and $\ker t$ are commute under the group operation.
\end{cor}

Another consequence of the interchange law is that one can give the inverse of a morphism in terms of group operation. That is, let $a\in G\left( x,y \right)$. Then
\[{{1}_{y}}=a\circ {{a}^{-1}}=a-{{1}_{x}}+{{a}^{-1}}.\]
Thus ${{a}^{-1}}={{1}_{x}}-a+{{1}_{y}}$. Similarly ${{a}^{-1}}={{1}_{y}}-a+{{1}_{x}}$.
A final remark is that if $a,{{a}_{1}}\in \ker s$ and $t\left( a \right)=x$ then $-{{1}_{x}}+a\in \ker t$ so commutes with ${{a}_{1}}$. This implies that
$\left( -{{1}_{x}}+a \right)+{{a}_{1}}={{a}_{1}}+\left( -{{1}_{x}}+a \right)$
and thus
\[a+{{a}_{1}}-a={{1}_{x}}+{{a}_{1}}-{{1}_{x}}.\]
\begin{thm}[Brown \& Spencer Theorem, \cite{BS76}]
	The category $\GpGd$ of group-groupoids and the category $\XMod$ of crossed modules are equivalent.
\end{thm}
\begin{proof}
	We sketch the proof since we need some details in the last section. Define a functor
	$\varphi :\text{\GpGd}\to \text{\XMod}$
	as follows: Let $G$ be a group-groupoid. Then $\varphi \left( G \right)=\left( A,B,\alpha  \right)$ is a crossed modules where $A=\ker s$, $B={{G}_{0}}$, $\alpha $ is the restriction of $t$ and the action of $B$ on $A$ is given by $x\cdot a={{1}_{x}}+a-{{1}_{x}}$.
	
	Conversely, define a functor
	$\psi :\text{\XMod}\to \text{\GpGd}$
	as follows: Let $\left( A,B,\alpha  \right)$ be a crossed module. Then the semi-direct product group $A\rtimes B$ is a group-groupoid on $B$ where $s\left( a,b \right)=b$, $t\left( a,b \right)=\alpha \left( a \right)+b$, $\varepsilon \left( b \right)=\left( 0,b \right)$ and the composition is
	$\left( a',b' \right)\circ \left( a,b \right)=\left( a'+a,b \right)$
	where $b'=\alpha \left( a \right)+b$.
	
	Other details are straightforward so is omitted.
\end{proof}

\section{Internal categories within the category of crossed modules}

In this section we will characterize internal categories in the category $\XMod$. Let $C$ be an internal category in the category $\XMod$ of crossed modules over groups. Then $C$ consist of two crossed modules ${{C}_{1}}=\left( {{A}_{1}},{{B}_{1}},{{\alpha }_{1}} \right)$ and ${{C}_{0}}=\left( {{A}_{0}},{{B}_{0}},{{\alpha }_{0}} \right)$ and four crossed module morphisms as $s=\left\langle {{s}_{A}},{{s}_{B}} \right\rangle$, $t=\left\langle {{t}_{A}},{{t}_{B}} \right\rangle :{{C}_{1}}\to {{C}_{0}}$  which are called the source and the target maps respectively, $\varepsilon =\left\langle {{\varepsilon }_{A}},{{\varepsilon }_{B}} \right\rangle :{{C}_{0}}\to {{C}_{1}}$ which is called the identity object map and $m=\left\langle {{m}_{A}},{{m}_{B}} \right\rangle :{{C}_{1}}\,{}_{s}{{\times }_{t}}\,{{C}_{1}}\to {{C}_{1}}$ which is called the composition map. These are object to the followings:
\begin{enumerate}[label={\textbf{(\roman{*})}}]
	\item $s\varepsilon =t\varepsilon ={{1}_{{{C}_{0}}}}$;
	\item $sm=s{{\pi }_{2}}$, $tm=t{{\pi }_{1}}$;
	\item $m\left( {{1}_{{{C}_{1}}}}\times m \right)=m\left( m\times {{1}_{{{C}_{1}}}} \right)$ and
	\item $m\left( \varepsilon s,{{1}_{{{C}_{1}}}} \right)=m\left( {{1}_{{{C}_{1}}}},\varepsilon t \right)={{1}_{{{C}_{1}}}}$.
\end{enumerate}

\[\xymatrix@R=10mm@C=10mm{
	A_1 {_{s_{A}}\times_{t_{A}}}  A_1 \ar[r]^-{m_{A}} \ar[d]_-{\alpha_1 \times \alpha_1} & A_1 \ar[d]_-{\alpha_1} \ar@<.3ex>[r]^-{s_{A}} \ar@<-.3ex>[r]_-{t_{A}}  & A_0 \ar[d]^-{\alpha_0}  \ar@/_/@<-1ex>[l]_-{\varepsilon_A} \\
	B_1 {_{s_{B}}\times_{t_{B}}}  B_1 \ar[r]_-{m_{B}} & B_1  \ar@<.3ex>[r]^-{s_{B}} \ar@<-.3ex>[r]_-{t_{B}}  & B_0  \ar@/^/@<1ex>[l]^-{\varepsilon_B}	
}\]

An internal category in the category $\XMod$ will be denoted by $C=\left( {{C}_{1}},{{C}_{0}},s,t,\varepsilon ,m \right)$ or briefly by $C$ when no confusion arise. Identity objects ${{\varepsilon }_{A}}\left( {{a}_{0}} \right)$ and ${{\varepsilon }_{B}}\left( {{b}_{0}} \right)$ will be denoted by ${{1}_{{{a}_{0}}}}$ and ${{1}_{{{b}_{0}}}}$  for short, respectively. Also the composition of elements will be denoted by ${{m}_{A}}\left( {{a}_{1}},{{a}_{1}}' \right)={{a}_{1}}\circ {{a}_{1}}'$ and by
${{m}_{B}}\left( {{b}_{1}},{{b}_{1}}' \right)={{b}_{1}}\circ {{b}_{1}}'$ for ${{a}_{1}},{{a}_{1}}'\in {{A}_{1}}$ and ${{b}_{1}},{{b}_{1}}'\in {{B}_{1}}$ with ${{s}_{A}}\left( {{a}_{1}} \right)={{t}_{A}}\left( {{a}_{1}}' \right)$ and ${{s}_{B}}\left( {{b}_{1}} \right)={{t}_{B}}\left( {{b}_{1}}' \right)$.

\begin{ex}
	Let $(A,B,\alpha)$ be a crossed module over groups. We know that $(A\times A,B\times B,\alpha\times\alpha)$ is also a crossed module. If we set $C_1=(A\times A,B\times B,\alpha\times\alpha)$, $C_0=(A,B,\alpha)$, $s=\pi_1$, $t=\pi_2$, $\varepsilon=\Delta$ and define $m$ with $(a_1,a_2)\circ(a,a_1)=(a,a_2)$ and $(b_1,b_2)\circ(b,b_1)=(b,b_2)$ for all $a,a_1,a_2\in A$ and $b,b_1,b_2\in B$ then $C=(C_1,C_0,s,t,\varepsilon,m)$ becomes an internal category in $\XMod$.
\end{ex}

\begin{ex}
	Let $(A,B,\alpha)$ be a crossed module over groups. Then $C=((A,B,\alpha),(A,B,\alpha),s,t,\varepsilon,m)$ becomes an internal category in $\XMod$ where $s$, $t$ and $\varepsilon$ are identity maps.
\end{ex}

\begin{ex}
	Let $(A,B,\alpha)$ be a topological crossed module. Then $(\pi A,\pi B, \pi \alpha)$ is also a crossed module. Moreover, $\pi(A,B,\alpha)=((\pi A,\pi B, \pi \alpha),(A,B,\alpha),s,t,\varepsilon,m)$ is an internal category in $\XMod$.
\end{ex}

Now we will give the properties of an internal category with a few lemmas individually.

\begin{lem}
	Let $C$ be an internal category in $\XMod$. Then for $i\in \left\{ 0,1 \right\}$
	\begin{enumerate}[label={\textbf{(\roman{*})}}]
		\item $\alpha_{i}({{a}_{i}}+{{a}_{i}}')=\alpha_{i}({{a}_{i}})+\alpha_{i}({{a}_{i}}')$,
		\item ${{\alpha }_{i}}\left( {{b}_{i}}\cdot {{a}_{i}} \right)={{b}_{i}}+{{\alpha }_{i}}\left( {{a}_{i}} \right)-{{b}_{i}}$ and
		\item ${{\alpha }_{i}}\left( {{a}_{i}} \right)\cdot {{a}_{i}}'={{a}_{i}}+{{a}_{i}}'-{{a}_{i}}$
	\end{enumerate}	
	for all ${{a}_{i}},{{a}_{i}}'\in {{A}_{i}}$ and ${{b}_{i}}\in {{B}_{i}}$.
\end{lem}
\begin{proof}
	It follows from the fact that ${{C}_{i}}=\left( {{A}_{i}},{{B}_{i}},{{\alpha }_{i}} \right)$ is a crossed module for $i\in \left\{ 0,1 \right\}$.
\end{proof}

\begin{lem}\label{catxmodproperties}
	Let $C$ be an internal category in $\XMod$. Then
	\begin{enumerate}[label={\textbf{(\roman{*})}}]
		\item\label{1} ${{s}_{A}}\left( {{a}_{1}}+{{a}_{1}}' \right)={{s}_{A}}\left( {{a}_{1}} \right)+{{s}_{A}}\left( {{a}_{1}}' \right)$, ${{s}_{B}}\left( {{b}_{1}}+{{b}_{1}}' \right)={{s}_{B}}\left( {{b}_{1}} \right)+{{s}_{B}}\left( {{b}_{1}}' \right)$,
		\\ ${{t}_{A}}\left( {{a}_{1}}+{{a}_{1}}' \right)={{t}_{A}}\left( {{a}_{1}} \right)+{{t}_{A}}\left( {{a}_{1}}' \right)$, ${{t}_{B}}\left( {{b}_{1}}+{{b}_{1}}' \right)={{t}_{B}}\left( {{b}_{1}} \right)+{{t}_{B}}\left( {{b}_{1}}' \right)$,
		\item\label{2} ${{\alpha }_{0}}{{s}_{A}}={{s}_{B}}{{\alpha }_{1}}$, ${{\alpha }_{0}}{{t}_{A}}={{t}_{B}}{{\alpha }_{1}}$,
		\item\label{3} ${{s}_{A}}\left( {{b}_{1}}\cdot {{a}_{1}} \right)={{s}_{B}}\left( {{b}_{1}} \right)\cdot {{s}_{A}}\left( {{a}_{1}} \right)$, ${{t}_{A}}\left( {{b}_{1}}\cdot {{a}_{1}} \right)={{t}_{B}}\left( {{b}_{1}} \right)\cdot {{t}_{A}}\left( {{a}_{1}} \right)$,
		\item\label{4} ${{\varepsilon }_{A}}\left( {{a}_{0}}+{{a}_{0}}' \right)={{\varepsilon }_{A}}\left( {{a}_{0}} \right)+{{\varepsilon }_{A}}\left( {{a}_{0}}' \right)$,${{\varepsilon }_{B}}\left( {{b}_{0}}+{{b}_{0}}' \right)={{\varepsilon }_{B}}\left( {{b}_{0}} \right)+{{\varepsilon }_{B}}\left( {{b}_{0}}' \right)$,
		\item\label{5} ${{\alpha }_{1}}{{\varepsilon }_{A}}={{\varepsilon }_{B}}{{\alpha }_{0}}$,
		\item\label{6} ${{\varepsilon }_{A}}\left( {{b}_{0}}\cdot {{a}_{0}} \right)={{\varepsilon }_{B}}\left( {{b}_{0}} \right)\cdot {{\varepsilon }_{A}}\left( {{a}_{0}} \right)$,
		\item\label{7} $\left( {{a}_{1}}+{{a}_{1}'} \right)\circ \left( {{a}_{1}''}+{{a}_{1}'''} \right)=\left( {{a}_{1}}\circ {{a}_{1}''} \right)+\left( {{a}_{1}}'\circ {{a}_{1}}''' \right)$ with ${{s}_{A}}\left( {{a}_{1}} \right)={{t}_{A}}\left( {{a}_{1}}'' \right)$ and ${{s}_{A}}\left( {{a}_{1}}' \right)={{t}_{A}}\left( {{a}_{1}}''' \right)$,\\ $\left( {{b}_{1}}+{{b}_{1}}' \right)\circ \left( {{b}_{1}}''+{{b}_{1}}''' \right)=\left( {{b}_{1}}\circ {{b}_{1}}'' \right)\circ \left( {{b}_{1}}'\circ {{b}_{1}}''' \right)$ with
		${{s}_{B}}\left( {{b}_{1}} \right)={{t}_{B}}\left( {{b}_{1}}'' \right)$ and ${{s}_{B}}\left( {{b}_{1}}' \right)={{t}_{B}}\left( {{b}_{1}}''' \right)$,
		\item\label{8} ${{\alpha }_{1}}{{m}_{A}}={{m}_{B}}\left( {{\alpha }_{1}}\times {{\alpha }_{1}} \right)$,
		\item\label{9} $\left( {{b}_{1}}\circ {{b}_{1}}' \right)\cdot \left( {{a}_{1}}\circ {{a}_{1}}' \right)=\left( {{b}_{1}}\cdot {{a}_{1}} \right)\circ \left( {{b}_{1}}'\cdot {{a}_{1}}' \right)$
	\end{enumerate}	
	for all ${{a}_{1}},{{a}_{1}'}\text{,}{{a}_{1}''},{{a}_{1}'''}\in {{A}_{1}}, {{b}_{1}},{{b}_{1}'}\text{,}{{b}_{1}''},{{b}_{1}'''}\in {{B}_{1}}$, ${{a}_{0}},{{a}_{0}}'\in {{A}_{0}}$, and ${{b}_{0}},{{b}_{0}}'\in {{B}_{0}}$.
\end{lem}

\begin{proof}
	\ref{1}-\ref{3} follows from the fact that $s=\left\langle {{s}_{A}},{{s}_{B}} \right\rangle $, $t=\left\langle {{t}_{A}},{{t}_{B}} \right\rangle $ being morphisms of crossed modules.\\
	\ref{4}-\ref{6} follows from the fact that $\varepsilon =\left\langle {{\varepsilon }_{A}},{{\varepsilon }_{B}} \right\rangle $ being a morphism of crossed modules. In these conditions if we use the symbol $\varepsilon \left( * \right)={{1}_{*}}$ for identity morphisms then we get
	\begin{enumerate}
		\item[\textit{\textbf{(iv)'}}] ${{1}_{{{a}_{0}}+{{a}_{0}}'}}={{1}_{{{a}_{0}}}}+{{1}_{{{a}_{0}}'}}$, ${{1}_{{{b}_{0}}+{{b}_{0}}'}}={{1}_{{{b}_{0}}}}+{{1}_{{{b}_{0}}'}}$,
		\item[\textit{\textbf{(v)'}}] ${{\alpha }_{1}}\left( {{1}_{{{a}_{0}}}} \right)={{1}_{{{\alpha }_{0}}\left( {{a}_{0}} \right)}}$  and
		\item[\textit{\textbf{(vi)'}}] ${{1}_{{{b}_{0}}\cdot {{a}_{0}}}}={{1}_{{{b}_{0}}}}\cdot {{1}_{{{a}_{0}}}}$.
	\end{enumerate}		
	\ref{7}-\ref{9} follows from the fact that $m=\left\langle {{m}_{A}},{{m}_{B}} \right\rangle $ being a morphism of crossed modules.
\end{proof}

The identities given in condition \ref{7} of Lemma \ref{catxmodproperties} are called interchange laws between group operations and compositions. As an application of interchange laws we will give the following corollary.

\begin{cor}
	Let $C$ be an internal category in $\XMod$. Then the compositions in
	${{A}_{1}}$ and ${{B}_{1}}$ can be written in terms of group operations on ${{A}_{1}}$ and ${{B}_{1}}$, respectively, as
	\[{{a}_{1}}\circ {{a}_{1}}'={{a}_{1}}-{{1}_{{{s}_{A}}\left( {{a}_{1}} \right)}}+{{a}_{1}}'={{a}_{1}}'-{{1}_{{{s}_{A}}\left( {{a}_{1}} \right)}}+{{a}_{1}}\]
	and
	\[{{b}_{1}}\circ {{b}_{1}}'={{b}_{1}}-{{1}_{{{s}_{B}}\left( {{b}_{1}} \right)}}+{{b}_{1}}'={{b}_{1}}'-{{1}_{{{s}_{B}}\left( {{b}_{1}} \right)}}+{{b}_{1}}\]
	for ${{a}_{1}},{{a}_{1}}'\in A$, ${{b}_{1}},{{b}_{1}}'\in B$ with ${{s}_{A}}\left( {{a}_{1}} \right)={{t}_{A}}\left( {{a}_{1}}' \right)$ and ${{s}_{B}}\left( {{b}_{1}} \right)={{t}_{B}}\left( {{b}_{1}}' \right)$.
\end{cor}

\begin{proof}
	We will prove the assumption for ${{A}_{1}}$. If $0$ denotes the identity (zero) elements of groups ${{A}_{1}}$ and ${{A}_{0}}$ then
	\[\begin{array}{rl}
	{{a}_{1}}\circ {{a}_{1}}' & =\left( {{a}_{1}}+0 \right)\circ \left( {{1}_{{{s}_{A}}\left( {{a}_{1}} \right)}}+\left( -{{1}_{{{s}_{A}}\left( {{a}_{1}} \right)}}+{{a}_{1}}' \right) \right) \\
	& =\left( {{a}_{1}}\circ {{1}_{{{s}_{A}}\left( {{a}_{1}} \right)}} \right)+\left( 0\circ \left( -{{1}_{{{s}_{A}}\left( {{a}_{1}} \right)}}+{{a}_{1}}' \right) \right) \\
	& ={{a}_{1}}-{{1}_{{{s}_{A}}\left( {{a}_{1}} \right)}}+{{a}_{1}}'
	\end{array}\]
	and similarly
	\[\begin{array}{rl}
	{{a}_{1}}\circ {{a}_{1}}' & =\left( 0+{{a}_{1}} \right)\circ \left( \left( {{a}_{1}}'-{{1}_{{{s}_{A}}\left( {{a}_{1}} \right)}} \right)+{{1}_{{{s}_{A}}\left( {{a}_{1}} \right)}} \right) \\
	& =\left( 0\circ \left( {{a}_{1}}'-{{1}_{{{s}_{A}}\left( {{a}_{1}} \right)}} \right) \right)+\left( {{a}_{1}}\circ {{1}_{{{s}_{A}}\left( {{a}_{1}} \right)}} \right) \\
	& ={{a}_{1}}'-{{1}_{{{s}_{A}}\left( {{a}_{1}} \right)}}+{{a}_{1}}\,.
	\end{array}\]
\end{proof} 

By this corollary we obtain that if ${{s}_{A}}\left( {{a}_{1}} \right)={{t}_{A}}\left( {{a}_{1}}' \right)=0$, i.e. ${{a}_{1}}\in \ker {{s}_{A}}$  and ${{a}_{1}}'\in \ker {{t}_{A}}$, then
\[{{a}_{1}}+{{a}_{1}}'={{a}_{1}}'+{{a}_{1}}.\]

So the elements of $\ker {{s}_{A}}$ and $\ker {{t}_{A}}$ are commutative. Similarly, the elements of $\ker {{s}_{B}}$ and $\ker {{t}_{B}}$ are commutative too. Moreover, for an element ${{a}_{1}}\in {{A}_{1}}$, ${{a}_{1}}^{-1}={{1}_{{{s}_{A}}\left( {{a}_{1}} \right)}}-{{a}_{1}}+{{1}_{{{t}_{A}}\left( {{a}_{1}} \right)}}\in {{A}_{1}}$ is the inverse element of ${{a}_{1}}$ up to the composition ${{m}_{A}}$. Similarly for an element ${{b}_{1}}\in {{B}_{1}}$, ${{b}_{1}}^{-1}={{1}_{{{s}_{B}}\left( {{b}_{1}} \right)}}-{{b}_{1}}+{{1}_{{{t}_{B}}\left( {{b}_{1}} \right)}}\in {{B}_{1}}$ is the inverse element of ${{b}_{1}}$ up to the composition ${{m}_{B}}$. This means that $C=\left( {{C}_{1}},{{C}_{0}},s,t,\varepsilon ,m,n \right)$ has a groupoid structure where $n=\left\langle {{n}_{A}},{{n}_{B}} \right\rangle :{{C}_{1}}\to {{C}_{1}}$ is a morphism of crossed modules with
\[\begin{array}{*{35}{l}}
{{n}_{A}} & : & {{A}_{1}} & \to  & {{A}_{1}}  \\
{} & {} & {{a}_{1}} & \mapsto  & {{n}_{A}}\left( {{a}_{1}} \right)=a_{1}^{-1}={{1}_{{{s}_{A}}\left( {{a}_{1}} \right)}}-{{a}_{1}}+{{1}_{{{t}_{A}}\left( {{a}_{1}} \right)}}  \\
\end{array}\]
and
\[\begin{array}{*{35}{l}}
{{n}_{B}} & : & {{B}_{1}} & \to  & {{B}_{1}}  \\
{} & {} & {{b}_{1}} & \mapsto  & {{n}_{B}}\left( {{b}_{1}} \right)=b_{1}^{-1}={{1}_{{{s}_{B}}\left( {{b}_{1}} \right)}}-{{b}_{1}}+{{1}_{{{t}_{B}}\left( {{b}_{1}} \right)}}\,.  \\
\end{array}\]

It is easy to see that ${{1}_{{{s}_{A}}\left( {{a}_{1}} \right)}}-{{a}_{1}}+{{1}_{{{t}_{A}}\left( {{a}_{1}} \right)}}={{1}_{{{t}_{A}}\left( {{a}_{1}} \right)}}-{{a}_{1}}+{{1}_{{{s}_{A}}\left( {{a}_{1}} \right)}}$ for all ${{a}_{1}}\in {{A}_{1}}$ and similarly ${{1}_{{{s}_{B}}\left( {{b}_{1}} \right)}}-{{b}_{1}}+{{1}_{{{t}_{B}}\left( {{b}_{1}} \right)}}={{1}_{{{t}_{B}}\left( {{b}_{1}} \right)}}-{{b}_{1}}+{{1}_{{{s}_{B}}\left( {{b}_{1}} \right)}}$ for all ${{b}_{1}}\in {{B}_{1}}$.

\begin{lem}
	Let ${{a}_{1}}\in {{A}_{1}}$ and ${{b}_{1}}\in {{B}_{1}}$. Then $b_{1}^{-1}\cdot a_{1}^{-1}={{\left( {{b}_{1}}\cdot {{a}_{1}} \right)}^{-1}}$.
\end{lem}

\begin{proof}
	By the condition \ref{9} of Lemma \ref{catxmodproperties}
	\[\begin{array}{rl}
	\left( {{b}_{1}}\cdot {{a}_{1}} \right)\circ \left( b_{1}^{-1}\cdot a_{1}^{-1} \right)& =\left( {{b}_{1}}\circ b_{1}^{-1} \right)\cdot \left( {{a}_{1}}\circ a_{1}^{-1} \right) \\
	& ={{1}_{s_B({{b}_{1}})}}\cdot {{1}_{s_A({{a}_{1}})}} \\
	& ={{1}_{s_B({{b}_{1}})\cdot s_A({{a}_{1}})}} \\
	& ={{1}_{s_A({{b}_{1}}\cdot {{a}_{1}})}}
	\end{array}\]
	and similarly $\left(b_{1}^{-1}\cdot a_{1}^{-1} \right)\circ \left( {{b}_{1}}\cdot {{a}_{1}} \right)={{1}_{t_A({{b}_{1}}\cdot {{a}_{1}})}}$. Thus $b_{1}^{-1}\cdot a_{1}^{-1}={{\left( {{b}_{1}}\cdot {{a}_{1}} \right)}^{-1}}$.
\end{proof}

It is easy to see that an internal category in the category of crossed modules over groups is indeed a crossed module object in the category of internal categories within groups.

\begin{defn}
	Let $C$ and $C'$ be two internal categories in $\XMod$. A morphism (internal functor) from $C$ to $C'$ is a pair of crossed module morphisms
	$f=\left( {{f}_{1}}=\left\langle f_{1}^{A},f_{1}^{B} \right\rangle ,{{f}_{0}}=\left\langle f_{0}^{A},f_{0}^{B} \right\rangle  \right):C\to C'$
	such that ${{f}_{0}}s=s{{f}_{1}}$, ${{f}_{0}}t=t{{f}_{1}}$, ${{f}_{1}}\varepsilon =\varepsilon {{f}_{0}}$ and ${{f}_{1}}m=m\left( {{f}_{1}}\times {{f}_{1}} \right)$.
\end{defn}

Hence we can construct the category of internal categories (groupoids) within the category of crossed modules over groups where the morphisms are internal functors as defined above. This category will be denoted by $\Cat(\XMod)$.

\subsection{Crossed squares}

Crossed squares are first defined in \cite{GWLod81}. In this subsection we recall the definition of a crossed square as given in \cite{BL87}. Further we prove that the category of crossed squares and of internal categories within the crossed modules are equivalent. Finally we give some examples of crossed squares using this equivalence.

\begin{defn}\cite{BL87}\label{xsq}
	A crossed square over groups consists of four morphisms of groups $\lambda :L\to M$, ${\lambda }':L\to N$, $\mu :M\to P$ and $\nu :N\to P$, such that $\nu {\lambda }'=\mu \lambda $ together with actions of the group $P$ on $L$, $M$, $N$ on the left, conventionally, (and hence actions of $M$ on $L$ and $N$ via $\mu $ and of $N$ on $L$ and $M$ via $\nu $) and a function $h:M\times N\to L$. These are subject to the following axioms:
	\begin{enumerate}[label={\textbf{(\roman{*})}}]
		\item $\lambda $, ${\lambda }'$ are $P$-equivariant and $\mu $, $\nu $ and $\kappa =\mu \lambda $ are crossed modules,
		\item $\lambda h(m,n)=m+n\cdot (-m)$, ${\lambda }'h(m,n)=m\cdot n-n$,
		\item $h(\lambda (l),n)=l+n\cdot (-l)$, $h(m,{\lambda }'(l))=m\cdot l-l$,
		\item $h(m+{m}',n)=m\cdot h({m}',n)+h(m,n)$, $h(m,n+{n}')=h(m,n)+n\cdot h(m,{n}')$,
		\item $h(p\cdot m,p\cdot n)=p\cdot h(m,n)$
	\end{enumerate}
	for all $l\in L$, $m,{m}'\in M$, $n,{n}'\in N$ and $p\in P$.
\end{defn}
\[\xymatrix{
	L \ar[r]^\lambda \ar[d]_{\lambda'}  & M \ar[d]^{\mu} \\
	N  \ar[r]_{\nu}  & P}\]
A crossed square will be denoted by $S=\left( L,M,N,P \right)$.

\begin{ex}\cite{Nor87,Nor90}
	Let $(A,B,\alpha )$ be crossed module and $(S,T,\sigma )$ a normal subcrossed module of $(A,B,\alpha )$. Then
	\[\xymatrix{
		S \ar@{^{(}->}[d]_{\operatorname{inc}} \ar[r]^{\sigma}  & T \ar@{^{(}->}[d]^{\operatorname{inc}}  \\
		A  \ar[r]_{\partial}  & B}\]
	forms a crossed square of groups where the action of $B$ on $S$ is induced action from the action of $B$ on $A$ and the action of $B$ on $T$ is conjugation. The h map is defined by $h(t,a)=t\cdot a - a$	for all $t\in T$ and $a\in A$.
\end{ex}
A topological example of crossed squares is the fundamental crossed square which is defined in \cite{BL87} as follows: Suppose given a commutative square of spaces
\[\xymatrix{
	C \ar[r]^f \ar[d]_{g}  & A \ar[d]^{a} \\
	B  \ar[r]_{b}  & X}\]
Let $F(f)$ be the homotopy fibre of $f$ and $F(X)$ the homotopy fibre of $F(g)\rightarrow F(a)$. Then the commutative square of groups
\[\xymatrix{
	\pi_1F(\bm X) \ar[r] \ar[d]  & \pi_1F(g) \ar[d] \\ \pi_1F(f)  \ar[r]  & \pi_1(C)}\]
is naturally equipped with a structure of crossed square. This crossed square is called the fundamental crossed square \cite{BL87}.

A morphism $f=\left( {{f}_{L}},{{f}_{M}},{{f}_{N}},{{f}_{P}} \right)$ of crossed squares from ${{S}_{1}}=\left( {{L}_{1}},{{M}_{1}},{{N}_{1}},{{P}_{1}} \right)$ to ${{S}_{2}}=\left( {{L}_{2}},{{M}_{2}},{{N}_{2}},{{P}_{2}} \right)$ consist of four group homomorphisms ${{f}_{L}}:{{L}_{1}}\to {{L}_{2}}$, ${{f}_{M}}:{{M}_{1}}\to {{M}_{2}}$, ${{f}_{N}}:{{N}_{1}}\to {{N}_{2}}$ and ${{f}_{P}}:{{P}_{1}}\to {{P}_{2}}$ which are compatible with the actions and the functions ${{h}_{1}}$ and ${{h}_{2}}$.

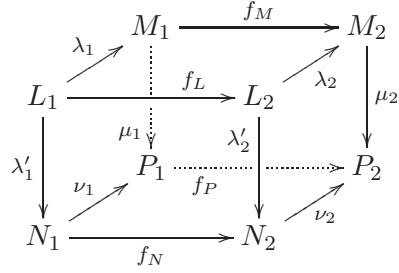
\[\xymatrix@R=4mm@C=7mm{
	& M_1 \ar@{..>}'[d][dd]_{\mu_1} \ar[rr]^{f_M} & & M_2 \ar[dd]^{\mu_2}  \\
	L_1 \ar[dd]_{\lambda'_1} \ar[rr]^(.7){f_L} \ar[ur]^{\lambda_1} & & L_2 \ar[dd]_(.3){\lambda'_2} \ar[ur]_{\lambda_2} \\
	& P_1 \ar@{..>}'[r]_{f_P}[rr] & & P_2 \\
	N_1 \ar[rr]_{f_N} \ar[ur]^{\nu_1} & & N_2 \ar[ur]_{\nu_2} \\
}\]

Category of crossed squares over groups with morphisms between crossed squares defined above is denoted by $\XSq$. Crossed squares are equivalent to the crossed modules over crossed modules \cite{Nor90}.

Now we prove that the category $\Cat(\XMod)$ of internal categories within the category of crossed modules over groups and the category $\XSq$ of crossed squares over groups are equivalent.

\begin{thm}
	The category $\Cat(\XMod)$ of internal categories within the category of crossed modules over groups and the category $\XSq$ of crossed squares over groups are equivalent.
\end{thm}

\begin{proof}
	We first define a functor $\eta :\Cat(\XMod)\to \XSq$ as follows: Let $C=\left( {{C}_{1}},{{C}_{0}},s,t,\varepsilon ,m,n \right)$ be an object in $\Cat(\XMod)$. If we set $L=\ker {{s}_{A}}$, $M=\ker {{s}_{B}}$, $N={{A}_{0}}$, $P={{B}_{0}}$, $\lambda ={{\alpha }_{1|\ker {{s}_{A}}}}$, $\lambda '={{t}_{A|\ker {{s}_{A}}}}$, $\mu ={{t}_{B|\ker {{s}_{B}}}}$ and $\nu ={{\alpha }_{0}}$ then $\eta \left( C \right)=S=\left( L,M,N,P \right)$ becaomes a crossed square with the function $h(m,n)=m\cdot {{1}_{n}}-{{1}_{n}}$ for all $m\in M$ and $n\in N$.
	Here $\left( L,M,\lambda  \right)$ is a crossed module since it is the kernel crossed module of \[s=\left\langle {{s}_{A}},{{s}_{B}} \right\rangle :\left( {{A}_{1}},{{B}_{1}},{{\alpha }_{1}} \right)\to \left( {{A}_{0}},{{B}_{0}},{{\alpha }_{0}} \right).\]
	Moreover we know that  $\left( L,N,\lambda ' \right)$ and $\left( M,P,\mu  \right)$ are crossed modules by Brown \& Spencer Theorem. Finally $\left( N,P,\nu  \right)$ is already a crossed module since it is $\left( {{A}_{0}},{{B}_{0}},{{\alpha }_{0}} \right)$. Here the actions of $P$ on $N$ is already given, on $M$ is given by $p\cdot m={{1}_{p}}+m-{{1}_{p}}$ and on $L$ is given by $p\cdot l={{1}_{p}}\cdot l$ (where the action on the right of the equation is the action of ${{B}_{1}}$ on ${{A}_{1}}$) for $p\in P$, $m\in M$ and $l\in L$. Now we need to show that the conditions given in the Definition \ref{xsq} is satisfied.
	\begin{enumerate}[label={\textbf{(\roman{*})}}]
		\item We need to show that $\lambda $, ${\lambda }'$ are $P$-equivariant and $\kappa =\mu \lambda $ is a crossed module. Let $l\in L$, and $p\in P$. Then
		\[\lambda (p\cdot l)={{\alpha }_{1}}({{1}_{p}}\cdot l)={{1}_{p}}+{{\alpha }_{1}}(l)-{{1}_{p}}=p\cdot \lambda (l)\]
		and
		\[{\lambda }'(p\cdot l)={{t}_{A}}({{1}_{p}}\cdot l)={{t}_{B}}\left( {{1}_{p}} \right)\cdot {{t}_{A}}(l)=p\cdot {\lambda }'(l)\]
		so $\lambda $ and ${\lambda }'$ are $P$-equivariant. Now we need to show that $\left( L,P,\kappa  \right)$ is a crossed module. So
		\begin{enumerate}[label={\textbf{(CM\arabic{*})}},leftmargin=1.5cm]
			\item Let $l\in L$, and $p\in P$. Then
			\[\begin{array}{rl}
			\kappa \left( p\cdot l \right) & =\mu \lambda \left( p\cdot l \right) \\
			& =\mu \left( p\cdot \lambda \left( l \right) \right) \\
			& =p+\mu \left( \lambda \left( l \right) \right)-p \\
			& =p+\kappa \left( l \right)-p
			\end{array}\]
			\item Let $l,l'\in L$. Then
			\[\begin{array}{rl}
			\kappa \left( l \right)\cdot l' & =\mu \left( \lambda \left( l \right) \right)\cdot l' \\
			& ={{1}_{\mu \left( \lambda \left( l \right) \right)}}\cdot l' \\
			& ={{1}_{\nu \left( \lambda '\left( l \right) \right)}}\cdot l' \\
			& =\lambda \left( {{1}_{\lambda '\left( l \right)}} \right)\cdot l' \\
			& ={{1}_{\lambda '\left( l \right)}}+l'-{{1}_{\lambda '\left( l \right)}} \\
			& =l+l'-l
			\end{array}\]
		\end{enumerate}	
		\item Let $m\in M$ and $n\in N$. Then
		\[\begin{array}{rl}
		\lambda h(m,n) & =\lambda \left( m\cdot {{1}_{n}}-{{1}_{n}} \right) \\
		& =\lambda \left( m\cdot {{1}_{n}} \right)-\lambda \left( {{1}_{n}} \right) \\
		& =m+\lambda \left( {{1}_{n}} \right)-m-\lambda \left( {{1}_{n}} \right) \\
		& =m+n\cdot \left( -m \right)
		\end{array}\]
		and
		\[\begin{array}{rl}
		{\lambda }'h(m,n)& ={\lambda }'\left( m\cdot {{1}_{n}}-{{1}_{n}} \right) \\
		& ={\lambda }'\left( m\cdot {{1}_{n}} \right)-{\lambda }'\left( {{1}_{n}} \right) \\
		& =\mu \left( m \right)\cdot {\lambda }'\left( {{1}_{n}} \right)-{\lambda }'\left( {{1}_{n}} \right) \\
		& =\mu \left( m \right)\cdot n-n \\
		& =m\cdot n-n
		\end{array}\]
		\item Let $l\in L$, $m\in M$ and $n\in N$. Then
		\[\begin{array}{rl}
		h(\lambda (l),n)& =\lambda (l)\cdot {{1}_{n}}-{{1}_{n}} \\
		& =\left( l+{{1}_{n}}-l \right)-{{1}_{n}} \\
		& =l+\left( {{1}_{n}}-l-{{1}_{n}} \right) \\
		& =l+n\cdot (-l)
		\end{array}\]
		and
		\[\begin{array}{rl}
		h\left( m,{\lambda }'\left( l \right) \right)& =m\cdot {{1}_{{\lambda }'\left( l \right)}}-{{1}_{{\lambda }'\left( l \right)}} \\
		& =\left( m\cdot {{1}_{{\lambda }'\left( l \right)}}-{{1}_{{\lambda }'\left( l \right)}}+l \right)-l \\
		& =\left( m\cdot {{1}_{{\lambda }'\left( l \right)}}\circ l \right)-l \\
		& =\left( \left( m\cdot {{1}_{{\lambda }'\left( l \right)}} \right)\circ \left( {{1}_{0}}\cdot l \right) \right)-l \\
		& =((m\circ {{1}_{0}})\cdot ({{1}_{{\lambda }'\left( l \right)}}\circ l))-l \\
		& =m\cdot l-l
		\end{array}\]
		\item Let $m,{m}'\in M$ and $n,{n}'\in N$. Then
		\[\begin{array}{rl}
		h\left( m+{m}',n \right)& =\left( m+{m}' \right)\cdot {{1}_{n}}-{{1}_{n}} \\
		& =m\cdot \left( {m}'\cdot {{1}_{n}} \right)-{{1}_{n}} \\
		& =m\cdot \left( {m}'\cdot {{1}_{n}} \right)+m\cdot \left( -{{1}_{n}}+{{1}_{n}} \right)-{{1}_{n}} \\
		& =m\cdot \left( {m}'\cdot {{1}_{n}}-{{1}_{n}} \right)+m\cdot {{1}_{n}}-{{1}_{n}} \\
		& =m\cdot h\left( {m}',n \right)+h\left( m,n \right)
		\end{array}\]
		and
		\[\begin{array}{rl}
		h\left( m,n+{n}' \right)& =m\cdot {{1}_{n+n'}}-{{1}_{n+n'}} \\
		& =m\cdot \left( {{1}_{n}}+{{1}_{n'}} \right)-{{1}_{n'}}-{{1}_{n}} \\
		& =\left( m\cdot {{1}_{n}}-{{1}_{n}} \right)+{{1}_{n}}+\left( m\cdot {{1}_{n'}}-{{1}_{n'}} \right)-{{1}_{n'}} \\
		& =h\left( m,n \right)+n\cdot h\left( m,{n}' \right)
		\end{array}\]
		\item Let $m\in M$, $n\in N$ and $p\in P$. Then
		\[\begin{array}{rl}
		& h\left( p\cdot m,p\cdot n \right)=h\left( {{1}_{p}}+m-{{1}_{p}},p\cdot n \right) \\
		& =\left( {{1}_{p}}+m-{{1}_{p}} \right)\cdot {{1}_{p\cdot n}}-{{1}_{p\cdot n}} \\
		& =\left( {{1}_{p}}+m-{{1}_{p}} \right)\cdot \left( {{1}_{p}}\cdot {{1}_{n}} \right)-\left( {{1}_{p}}\cdot {{1}_{n}} \right) \\
		& ={{1}_{p}}\cdot \left( m\cdot {{1}_{n}} \right)+{{1}_{p}}\cdot \left( -{{1}_{n}} \right) \\
		& ={{1}_{p}}\cdot \left( m\cdot {{1}_{n}}-{{1}_{n}} \right) \\
		& =p\cdot h\left( m,n \right)
		\end{array}\]
	\end{enumerate}
	
	Now let $f=\left( {{f}_{1}}=\left\langle f_{1}^{A},f_{1}^{B} \right\rangle ,{{f}_{0}}=\left\langle f_{0}^{A},f_{0}^{B} \right\rangle  \right):C\to C'$ be a morphism in $\Cat(\XMod)$. Then
	\[\eta \left( f \right)=\left( {{f}_{L}}=f_{1|\ker {{s}_{A}}}^{A},{{f}_{M}}=f_{1|\ker {{s}_{B}}}^{B},{{f}_{N}}=f_{0}^{A},{{f}_{P}}=f_{0}^{B} \right):S\to S'\]
	is a morphism of crossed squares.
	
	Conversely define a functor $\psi : \XSq\rightarrow \Cat(\XMod)$ as follows: Let $S=\left( L,M,N,P \right)$ be a crossed square over groups. Then \[\psi \left( S \right)=C=\left( {{C}_{1}}=\left( {{A}_{1}},{{B}_{1}},{{\alpha }_{1}} \right),{{C}_{0}}=\left( {{A}_{0}},{{B}_{0}},{{\alpha }_{0}} \right),s,t,\varepsilon ,m \right)\] is an internal category within the category of crossed modules over groups where
	$\left( {{A}_{1}},{{B}_{1}},{{\alpha }_{1}} \right)=\left( L\rtimes N,M\rtimes P,\lambda \times \nu  \right)$, $\left( {{A}_{0}},{{B}_{0}},{{\alpha }_{0}} \right)=\left( N,P,\nu  \right)$, ${{s}_{A}}\left( l,n \right)=n$, ${{s}_{B}}\left( m,p \right)=p$, ${{t}_{A}}\left( l,n \right)=\lambda '\left( l \right)+n$, ${{t}_{B}}\left( m,p \right)=\mu \left( m \right)+p$, ${{\varepsilon }_{A}}\left( n \right)=\left( 0,n \right)$, ${{\varepsilon }_{B}}\left( p \right)=\left( 0,p \right)$,
	\[\left( l',\lambda '\left( l \right)+n \right)\circ \left( l,n \right)=\left( l'+l,n \right)\]
	and
	\[\left( m',\lambda '\left( m \right)+p \right)\circ \left( m,p \right)=\left( m'+m,p \right).\]
	We know that ${{C}_{0}}$ is a crossed module over groups. First we need to show that $\left( L\rtimes N,M\rtimes P,\lambda \times \nu  \right)$ is a crossed module with the action of $M\rtimes P$ on $L\rtimes N$ is
	\[\left( m,p \right)\cdot \left( l,n \right)=\left( m\cdot \left( p\cdot l \right)+h\left( m,p\cdot n \right),p\cdot n \right).\]
	\begin{enumerate}[label={\textbf{(CM\arabic{*})}},leftmargin=1.5cm]
		\item Let $\left( l,n \right)\in L\rtimes N$ and $\left( m,p \right)\in M\rtimes P$. Then
		\[\begin{array}{rl}
		\left( \lambda \times \nu  \right)\left( \left( m,p \right)\cdot \left( l,n \right) \right) & =\left( \lambda \times \nu  \right)\left( m\cdot \left( p\cdot l \right)+h\left( m,p\cdot n \right),p\cdot n \right) \\
		& =\left( \lambda \left( m\cdot \left( p\cdot l \right)+h\left( m,p\cdot n \right) \right),\nu \left( p\cdot n \right) \right) \\
		& =\left( \lambda \left( m\cdot \left( p\cdot l \right) \right)+\lambda \left( h\left( m,p\cdot n \right) \right),p+\nu \left( n \right)-p \right) \\
		& =\left( m+\lambda \left( p\cdot l \right)-m+m+\left( p\cdot n \right)\cdot \left( -m \right),p+\nu \left( n \right)-p \right) \\
		& =\left( m+p\cdot \lambda \left( l \right)+\left( p\cdot n \right)\cdot \left( -m \right),p+\nu \left( n \right)-p \right) \\
		& =\left( m+p\cdot \lambda \left( l \right)+\left( p+\nu \left( n \right)-p \right)\cdot \left( -m \right),p+\nu \left( n \right)-p \right) \\
		& =\left( m+p\cdot \lambda \left( l \right)+\left( p+\nu \left( n \right) \right)\cdot \left( \left( -p \right)\cdot \left( -m \right) \right),p+\nu \left( n \right)-p \right) \\
		& =\left( m+p\cdot \lambda \left( l \right),p+\nu \left( n \right) \right)+\left( \left( -p \right)\cdot \left( -m \right),-p \right) \\
		& =\left( m,p \right)+\left( \lambda \left( l \right),\nu \left( n \right) \right)-\left( m,p \right) \\
		& =\left( m,p \right)+\left( \lambda \times \nu  \right)\left( l,n \right)-\left( m,p \right).
		\end{array}\]
		\item Let $\left( l,n \right),\left( l',n' \right)\in L\rtimes N$. Then
		\[\begin{array}{rl}
		\left( \lambda \times \nu  \right)\left( \left( l,n \right) \right)\cdot \left( l',n' \right) & =\left( \lambda \left( l \right),\nu \left( n \right) \right)\cdot \left( l',n' \right) \\
		& =\left( \lambda \left( l \right)\cdot \left( \nu \left( n \right)\cdot l' \right)+h\left( \lambda \left( l \right),\nu \left( n \right)\cdot n' \right),\nu \left( n \right)\cdot n' \right) \\
		& =\left( \lambda \left( l \right)\cdot \left( n\cdot l' \right)+h\left( \lambda \left( l \right),n+n'-n \right),n+n'-n \right) \\
		& =\left( l+n\cdot l'-l+l+\left( n+n'-n \right)\cdot \left( -l \right),n+n'-n \right) \\
		& =\left( l+n\cdot l'+\left( n+n' \right)\cdot \left( \left( -n \right)\cdot \left( -l \right) \right),n+n'-n \right) \\
		& =\left( l+n\cdot l',n+n' \right)+\left( \left( -n \right)\cdot \left( -l \right),-n \right) \\
		& =\left( l,n \right)+\left( l',n' \right)-\left( l,n \right).
		\end{array}\]
	\end{enumerate}
	Thus ${{C}_{1}}=\left( L\rtimes N,M\rtimes P,\lambda \times \nu  \right)$ is a crossed module. Now we need to show that $\psi \left( S \right)=C$ satisfies the conditions given in Lemma \ref{catxmodproperties} We know that ${{s}_{A}},{{s}_{B}},{{t}_{A}},{{t}_{B}},{{\varepsilon }_{A}},{{\varepsilon }_{B}},{{m}_{A}}$ and ${{m}_{B}}$ are group homomorphisms. So the conditions (i), (iv) and (vii) holds.
	\begin{enumerate}
		\item[\textbf{(ii)}] Let $\left( l,n \right)\in L\rtimes N$. Then
		\[\begin{array}{rl}
		\nu {{s}_{A}}\left( \left( l,n \right) \right)& =\nu \left( n \right) \\
		& ={{s}_{A}}\left( \lambda \left( l \right),\nu \left( n \right) \right) \\
		& ={{s}_{B}}\left( \left( \lambda \times \nu  \right)\left( l,n \right) \right)
		\end{array}\]
		and
		\[\begin{array}{rl}
		\nu {{t}_{A}}\left( l,n \right)& =\nu \left( \lambda '\left( l \right)+n \right) \\
		& =\nu \left( \lambda '\left( l \right) \right)+\nu \left( n \right) \\
		& =\mu \left( \lambda \left( l \right) \right)+\nu \left( n \right) \\
		& ={{t}_{B}}\left( \left( \lambda \left( l \right),\nu \left( n \right) \right) \right) \\
		& ={{t}_{B}}\left( \left( \lambda \times \nu  \right)\left( l,n \right) \right)
		\end{array}\]
		\item[\textbf{(iii)}] Let $\left( l,n \right)\in L\rtimes N$ and $\left( m,p \right)\in M\rtimes P$. Then
		${{s}_{A}}\left( \left( m,p \right)\cdot \left( l,n \right) \right)=p\cdot n={{s}_{B}}\left( m,p \right)\cdot {{s}_{A}}\left( l,n \right)$
		and
		\[\begin{array}{rl}
		{{t}_{A}}\left( \left( m,p \right)\cdot \left( l,n \right) \right)& =\lambda '\left( m\cdot \left( p\cdot l \right)+h\left( m,p\cdot n \right) \right)+p\cdot n \\
		& =\lambda '\left( m\cdot \left( p\cdot l \right) \right)+\lambda '\left( h\left( m,p\cdot n \right) \right)+p\cdot n \\
		& =\lambda '\left( \left( \mu \left( m \right)+p \right)\cdot l \right)+\left( m\cdot \left( p\cdot n \right)-p\cdot n \right)+p\cdot n \\
		& =\left( \mu \left( m \right)+p \right)\cdot \lambda '\left( l \right)+\left( \mu \left( m \right)+p \right)\cdot n \\
		& =\left( \mu \left( m \right)+p \right)\cdot \left( \lambda '\left( l \right)+n \right) \\
		& ={{t}_{B}}\left( m,p \right)\cdot {{t}_{A}}\left( l,n \right).
		\end{array}\]
		\item[\textbf{(v)}] Let $n\in N$. Then
		\[\begin{array}{rl}
		{{\alpha }_{1}}{{\varepsilon }_{A}}\left( n \right)& ={{\alpha }_{1}}\left( 0,n \right) \\
		& =\left( \lambda \times \nu  \right)\left( 0,n \right) \\
		& =\left( \lambda \left( 0 \right),\nu \left( n \right) \right) \\
		& ={{\varepsilon }_{B}}\nu \left( n \right) \\
		& ={{\varepsilon }_{B}}{{\alpha }_{0}}\left( n \right).
		\end{array}\]
		\item[\textbf{(vi)}] Let $n\in N$ and $p\in P$. Then
		\[\begin{array}{rl}
		{{\varepsilon }_{A}}\left( p\cdot n \right)& =\left( 0,p\cdot n \right) \\
		& =\left( 0,p \right)\cdot \left( 0,n \right) \\
		& ={{\varepsilon }_{B}}\left( p \right)\cdot {{\varepsilon }_{A}}\left( n \right).
		\end{array}\]
		\item[\textbf{(viii)}] Let $\left( l,n \right),\left( l',n' \right)\in L\rtimes N$ such that $n'=\lambda '\left( l \right)+n$. Then
		\[\begin{array}{rl}
		{{\alpha }_{1}}{{m}_{A}}\left( \left( l',n' \right),\left( l,n \right) \right)& ={{\alpha }_{1}}\left( \left( l',n' \right)\circ \left( l,n \right) \right) \\
		& =\left( \lambda \times \nu  \right)\left( \left( l'+l,n \right) \right) \\
		& =\left( \lambda \left( l'+l \right),\nu \left( n \right) \right) \\
		& =\left( \lambda \left( l' \right)+\lambda \left( l \right),\nu \left( n \right) \right) \\
		& =\left( \lambda \left( l' \right),\mu \left( \lambda \left( l \right) \right)+\nu \left( n \right) \right)\circ \left( \lambda \left( l \right),\nu \left( n \right) \right) \\
		& =\left( \lambda \left( l' \right),\nu \left( \lambda '\left( l \right) \right)+\nu \left( n \right) \right)\circ \left( \lambda \left( l \right),\nu \left( n \right) \right) \\
		& =\left( \lambda \left( l' \right),\nu \left( \lambda '\left( l \right)+n \right) \right)\circ \left( \lambda \left( l \right),\nu \left( n \right) \right) \\
		& =\left( \lambda \left( l' \right),\nu \left( n' \right) \right)\circ \left( \lambda \left( l \right),\nu \left( n \right) \right) \\
		& ={{m}_{B}}\left( \left( \lambda \times \nu  \right)\left( l',n' \right),\left( \lambda \times \nu  \right)\left( l,n \right) \right) \\
		& ={{m}_{B}}\left( {{\alpha }_{1}}\times {{\alpha }_{1}} \right)\left( \left( l',n' \right),\left( l,n \right) \right).
		\end{array}\]
		\item[\textbf{(ix)}] Let $\left( l,n \right),\left( l',n' \right)\in L\rtimes N$ and $\left( m,p \right),\left( m',p' \right)\in M\rtimes P$  such that $n'=\lambda '\left( l \right)+n$ and $p'=\mu \left( m \right)+p$. Firstly, \[\begin{array}{rl}
		h\left( m'+m,p\cdot n \right) & = h\left( m-m+m'+m,p\cdot n \right) \\
		& = m\cdot h\left( \left( -m \right)\cdot m',p\cdot n \right)+h\left( m,p\cdot n \right) \\
		& = h\left( m\cdot \left( \left( -m \right)\cdot m' \right),m\cdot \left( p\cdot n \right) \right)+h\left( m,p\cdot n \right) \\
		& = h\left( m',p'\cdot n \right)+h\left( m,p\cdot n \right).
		\end{array}\] Then $\left( \left( m',p' \right)\circ \left( m,p \right) \right)\cdot \left( \left( l',n' \right)\circ \left( l,n \right) \right) =\left( m'+m,p \right)\cdot \left( l'+l,n \right)$ and \[\small\begin{array}{rl}
		\left( m'+m,p \right)\cdot \left( l'+l,n \right)
		& =\left( \left( m'+m \right)\cdot \left( p\ cdot \left( l'+l \right) \right)+h\left( m'+m,p\cdot n \right),p\cdot n \right) \\
		& =\left( m'\cdot \left( p'\cdot l' \right)+m'\cdot \left( p'\cdot l \right)+h\left( m',p'\cdot n \right)+h\left( m,p\cdot n \right),p\cdot n \right) \\
		& =\left( m'\cdot \left( p'\cdot l' \right)+h\left( m',p'\cdot \lambda '\left( l \right)+p'\cdot n \right)+m\cdot \left( p\cdot l \right)+h\left( m,p\cdot n \right),p\cdot n \right) \\
		& =\left( m'\cdot \left( p'\cdot l' \right)+h\left( m',p'\cdot \left( \lambda '\left( l \right)+n \right) \right)+m\cdot \left( p\cdot l \right)+h\left( m,p\cdot n \right),p\cdot n \right) \\
		& =\left( m'\cdot \left( p'\cdot l' \right)+h\left( m',p'\cdot n' \right)+m\cdot \left( p\cdot l \right)+h\left( m,p\cdot n \right),p\cdot n \right) \\
		& =\left( m'\cdot \left( p'\cdot l' \right)+h\left( m',p'\cdot n' \right),p'\cdot n' \right)\circ \left( m\cdot \left( p\cdot l \right)+h\left( m,p\cdot n \right),p\cdot n \right) \\
		& =\left( \left( m',p' \right)\cdot \left( l',n' \right) \right)\circ \left( \left( m,p \right)\cdot \left( l,n \right) \right) \\
		\end{array}\]
	\end{enumerate}
	
	Thus $\psi \left( S \right)=C$ is an object in $\Cat(\XMod)$. Now let 
	\[f=\left( {{f}_{L}},{{f}_{M}},{{f}_{N}},{{f}_{P}} \right)\colon{{S}_{1}}=\left( {{L}_{1}},{{M}_{1}},{{N}_{1}},{{P}_{1}} \right)\rightarrow{{S}_{2}}=\left( {{L}_{2}},{{M}_{2}},{{N}_{2}},{{P}_{2}} \right)\]
	be a morphism of crossed squares. Then $\psi \left( f \right)=\left( \left\langle f_L\times f_N,f_M\times f_P\right\rangle, \left\langle f_N, f_P \right\rangle \right)\colon C\rightarrow C' $ is a morphism in $\Cat(\XMod)$ where $\psi(S_1)=C$ and $\psi(S_2)=C'$.
	
	Finally we show that composition of these functors are naturally isomorphic to the identity functors on $\Cat(\XMod)$ and $\XSq$ respectively. For any object $C$ in $\Cat(\XMod)$ the natural isomorphism $U\colon 1_{\Cat(\XMod)}\Rightarrow \psi\eta$ is given by $U_C=\left( f_1=\left\langle f_{1}^{A},f_{1}^{B} \right\rangle, f_0=\left\langle f_{0}^{A},f_{0}^{B} \right\rangle \right) $ where
	$f_{1}^{A}(a_1) = \left( a_1-1_{s_A(a_1)}, s_A(a_1) \right)$, $f_{1}^{B}(b_1) = \left( b_1-1_{s_B(b_1)}, s_B(b_1) \right)$, $f_{0}^{A}=1_{A_{0}}$ and $f_{0}^{B}=1_{B_{0}}$ for all $a_1\in A_1$ and $b_1\in B_1$.
	
	Conversely, for any object $S=(L,M,N,P)$ in $\XSq$ the natural isomorphism $T\colon\eta\psi\Rightarrow 1_{\XSq}$ is given by $T_S=\left( f_L=\pi_1, f_M=\pi_1, f_N=1_N, f_P=1_P \right) $. This completes the proof.
\end{proof}

Now we can give examples of crossed squares which are obtained from examples of internal categories within the category of crossed modules.

\begin{ex}
	Let $(A,B,\alpha)$ be a crossed module. Then the diagram
	\[\xymatrix{
		A \ar[r]^\alpha \ar[d]_{1}  & B \ar[d]^{1} \\
		A  \ar[r]_{\alpha}  & B}\]
	has a structure of a crossed square where $h(b,a)=b\cdot a-a$ for all $a\in A$ and $b\in B$.
\end{ex}

\begin{ex}
	Let $(A,B,\alpha)$ be a crossed module. Then the diagram
	\[\xymatrix{
		0 \ar[r]^0 \ar[d]_{0}  & A \ar[d]^{\alpha}  \\
		0 \ar[r]_{0}    & B}\]
	forms a crossed square where $h(a,0)=0$ for all $a\in A$.
\end{ex}

\begin{ex}
	Let $(A,B,\alpha)$ be a topological crossed module. Then we know that $(\pi A,A,s_A,t_A,\varepsilon_A,m_A)$ and $(\pi B,B,s_B,t_B,\varepsilon_B,m_B)$ are group-groupoids. Then
	\[\xymatrix{
		\ker s_A \ar[r]^{\pi\alpha} \ar[d]_{t_A}  & \ker s_B \ar[d]^{t_B}  \\
		A \ar[r]_{\alpha}    & B}\]
	has a crossed square structure where $h([\beta],a)=[\beta\cdot a-a]$ for all $[\beta]\in \ker s_B$ and $a\in A$. Here the path $(\beta\cdot a-a)\colon [0,1]\rightarrow A$ is given by $(\beta\cdot a-a)(r)=\beta(r)\cdot a-a$ for all $r\in [0,1]$.
\end{ex}

\section{Conclusion}

We proved that the category $\Cat(\XMod)$ of internal categories within the category of crossed modules over groups and the category $\XSq$ of crossed squares over groups are equivalent. Since crossed squares model all connected homotopy 3-types so are internal categories in within the category of crossed modules. 

For further work, in a similar way of thinking one can obtain same results in a more generic algebraic category namely the category of groups with operations or in higher dimensional crossed modules \cite{Ell88}. Also in the light of the results given in \cite{MSA15}, notions of normal subcrossed square and of quotient crossed square can be obtained.


\end{document}